\documentclass[letter,12pt]{article}
\usepackage{amssymb}
\usepackage[top=1in, bottom=1in, left=1in, right=1in]{geometry}
\usepackage{amsthm}
\usepackage{amsfonts}
\usepackage{amsmath}
\usepackage{bbm}
\usepackage{abstract}
\usepackage{color}

\newtheorem{theorem}{Theorem}[section]
\newtheorem{lemma}[theorem]{Lemma}
\newtheorem{prop}[theorem]{Proposition}

\newtheorem{conj}[theorem]{Conjecture}

\newtheorem{question}[theorem]{{Question}}

\theoremstyle{definition}

\newtheorem{ques}[theorem]{Question}

\newcommand{\PSH}{{\rm PSH}}

\usepackage{hyperref}
\hypersetup{
    bookmarks=true,         
    unicode=false,          
    pdftoolbar=true,       
    pdfmenubar=true,       
    pdffitwindow=false,    
    pdfstartview={FitH},      
    pdfauthor={T. Darvas},     
    colorlinks=true,       
   linkcolor=black,          
    citecolor=black,        
    filecolor=black,      
    urlcolor=black}           

\title{A decade of metric geometry in the space of K\"ahler metrics}
\author{Tam\'as Darvas}
\date{}
\begin{document}
\maketitle

\begin{abstract}We survey selected developments in the metric geometry of the space of Kähler metrics, emphasizing results from the past decade, highlighting open problems along the way.
\end{abstract}

\section{Introduction} 

Let $(X,\omega)$ be a compact connected K\"ahler manifold. This means that for every point $x \in X$ there exists a neighborhood $U$ and $g \in C^\infty(U)$ such that

$$\omega|_U = i \partial \bar \partial g = \sum_{j,k} i \frac{\partial^2 g}{\partial z_j \bar \partial z_k} dz_j \wedge \overline{d z_k} > 0.$$ 

This implies that $d \omega =0$, hence $\omega$ determines a de Rham class $[\omega] \in  H^2(X,\Bbb R).$  If $\omega'$ is another K\"ahler metric such that $[\omega'] = [\omega]$, then by the $\partial\bar \partial$-lemma of Hodge theory, there exists $v \in C^\infty(X)$ such that $\omega ' = \omega + i \partial \bar \partial v.$
Thus to study K\"ahler  metrics, one can conveniently study $\mathcal H_\omega \subset C^\infty(X)$, the space of K\"ahler potentials:
\[
\mathcal H_\omega := \{u \in C^\infty(X) \ |\  \omega_u := \omega + i \partial \bar \partial u > 0\}.
\]

The quest for a canonical representative in $\mathcal{H}_\omega$ has a long and rich history, originating from Calabi’s influential conjectures \cite{Cal} and continuing with Yau’s celebrated resolution \cite{Yau}. In connection with this program, it has become clear that one must endow $\mathcal{H}_\omega$ with special $L^p$-type Finsler structures, where the parameter $p$ depends on the geometric application considered \cite{Da15}:
\begin{equation}\label{eq: Finsler_eq}
\| \phi\|_{L^p} = \left(\int_X | \phi|^p \omega_u^n \right)^{1/p}, \quad u \in \mathcal H_\omega, \ \phi \in C^\infty(X) \simeq T_u \mathcal H_\omega.
\end{equation}

The case $p=2$ corresponds to the classical Riemannian metric of Mabuchi, Semmes, and Donaldson \cite{Ma,Se,Do}, and is central to the study of the Calabi flow \cite{CC, St, BDL1}, and uniqueness questions \cite{BaMa,Br15,BrBr}. The $p=1$ Finsler geometry has gained increasing relevance in connection with the study of K-stability and energy properness \cite{BBJ, BJ25, BDL2, CC1, CC2,DZ25, LiG}. Additionally, recent exciting work of Finski concerning  filtrations and spectral measures \cite{Fin2}, and spectral theory of Toeplitz operators \cite{Fin3}, uses the full spectrum of $L^p$ distances.

Several surveys explore the deep relationship between the Finsler geometries defined in \eqref{eq: Finsler_eq} and canonical metrics \cite{Da17s, Bl2, CCX, Gbook, PS,Sze_book, Ysurv}. The present note has much more modest goals: to describe selected developments in the metric geometry of $\mathcal H_\omega$ over the past decade, while highlighting a number of natural open problems that have received less attention that they perhaps deserve.

Owing to space constraints, several exciting directions are not covered, including intriguing developments related to K\"ahler quantization \cite{PS1,CS, br2, DLR}, the case of degenerate classes \cite{Da17m,DiG, DDL3,DiLu,Gup24, Tru22}, or works that cover other types of metric geometries on $\mathcal H_\omega$ \cite{Cala,ClRb, CalZh,Da24}. For a survey that develops the theory from basic principles, we refer the reader to \cite{Da17s}.

\paragraph{Acknowledgments.} We thank S. Boucksom, S. Finski, C. H. Lu, and Y. Wu for helpful discussions related to the topics of this survey. We are also grateful to the organizers of the 2025 ICBS Conference for their hospitality; much of this survey was written during the conference. AI tools were used for proofreading and to improve the clarity of the exposition.

\section{Weak geodesic segments} 

Let $[0,1] \ni t \mapsto u_t \in \mathcal H_\omega$ be a smooth path, in the sense that $v(t,x):= u_t(x) \in C^\infty([0,1] \times X)$. This curve is a geodesic if it has vanishing acceleration with respect to the $L^2$ Mabuchi metric: $\nabla_{\dot u_t} \dot u_t = 0$. As shown by Mabuchi \cite{Ma}, this is equivalent to:
\begin{equation}\label{eq: geod_eq_Lev_Civ}
\ddot u_t - \frac{1}{2}\langle \nabla^{\omega_{u_t}} \dot u_t , \nabla^{\omega_{u_t}} \dot u_t \rangle = 0, \quad t \in [0,1],
\end{equation}
where $\nabla^{\omega_{v}}$ denotes the gradient with respect to the K\"ahler metric $\omega_v$. This equation can be reinterpreted as a complex Monge–Ampère equation. To that end, we introduce the trivial complexification $u \in C^\infty(S \times X)$ defined by
\[
u(s,x) = u_{\textup{Re }s}(x),
\]
where $S := \{0 < \textup{Re }s <1 \} \subset \mathbb{C}$ is the open unit strip. After some algebraic manipulations, the geodesic equation becomes:
\begin{equation}\label{eq: geodesic_eq}
(\pi^*\omega + i \partial \overline{\partial}u)^{n+1} = 0 \quad \text{on } S \times X,
\end{equation}
where $\pi: S \times X \to X$ is the projection. Thus, finding a smooth geodesic joining $u_0, u_1 \in \mathcal H_\omega$ reduces to solving the following boundary value problem:
\begin{equation}\label{eq: BVPGeod}
\begin{cases}
(\pi^* \omega + i \partial \overline{\partial}u)^{n+1}=0, \\ 
\omega + i \partial \overline{\partial}u|_{\{s \}\times X} >0, \quad s \in S, \\ 
u(t+ir,x) = u(t,x), \quad \forall x \in X,\ t \in (0,1),\ r \in \mathbb{R},\\
\lim_{s \to 0}u(s,\cdot)=u_0, \quad \lim_{s \to 1}u(s,\cdot)=u_1.
\end{cases}
\end{equation}
Here, the limits $s \to 0,1$ are interpreted in the uniform $C^0$ topology.

In general, smooth solutions to this boundary problem do not exist, as recalled below. However, weak solutions (in the Bedford--Taylor sense \cite{BT}) with partial regularity are available, as we now describe.

Let $Y$ be a complex manifold equipped with a closed $(1,1)$-form $\eta$. The space of $\eta$-plurisubharmonic functions is denoted by $\PSH(Y,\eta)$. With slight abuse of precision, this space is defined as:
\[
\PSH(Y,\eta) := \{u : Y \to [-\infty,\infty) \mid u \text{ is usc and } \eta + i\partial \bar \partial u \geq 0 \text{ as  currents} \}.
\]

As observed by Berndtsson \cite[Section 2.1]{br2}, the (weak) solution $u$ of \eqref{eq: BVPGeod} can be expressed via a generalization of the classical Perron–Bremermann envelope from local theory. A key advantage of this approach is that it allows for rather general boundary data.
We refer to $i\mathbb{R}$-invariant functions in $\textup{PSH}(S \times X, \pi^*\omega)$ as \emph{weak subgeodesics}, with $S = \{0 < \textup{Re }z < 1\} \subset \mathbb{C}$. Such subgeodesics \( w \) are completely determined by their restrictions to \( (0,1) \times X \), which we may view as a curve \(  (0,1) \ni t \to w_t \in \textup{PSH}(X, \omega) \).

Let $u_0,u_1 \in \textup{PSH}(X,\omega)$. The corresponding Perron envelope is given by:
\begin{equation}\label{eq: udef1}
u = \sup_{v \in \mathcal S} v,
\end{equation}
where
\[
\mathcal S := \left\{ [0,1] \ni t \mapsto v_t \in \text{PSH}(X,\omega)\, \middle|\, v \text{ is a subgeodesic with } \lim_{t \to 0,1} v_t \leq u_{0,1} \right\}.
\]
Here the limits $\lim_{t \to 0,1} v_t$ are interpreted in the weak $L^1$ sense of $\omega$-psh functions. 
Due to convexity in $t$, every candidate subgeodesic satisfies $w_t \leq (1-t)u_0 + tu_1$, hence the resulting envelope $u$ from \eqref{eq: udef1} is  $\pi^*\omega$-psh on $ S \times X$.

We refer to the curve $t \mapsto u_t \in \textup{PSH}(X,\omega)$ obtained from \eqref{eq: udef1} as the \emph{weak geodesic} connecting $u_0$ and $u_1$. The following lemma confirms that this construction solves the boundary problem in the bounded setting:

\begin{lemma}[\cite{br2}]\label{lem: boundedBVP_solution}
Suppose $u_0,u_1 \in \textup{PSH}(X,\omega) \cap L^\infty$. Then the unique bounded $\pi^*\omega$-psh solution to \eqref{eq: BVPGeod} is given by the envelope \eqref{eq: udef1}.
\end{lemma}

Thus, when $u_0, u_1 \in \textup{PSH}(X,\omega) \cap L^\infty$, we refer to  $t \mapsto u_t$ as the \emph{bounded weak geodesic} connecting $u_0,u_1$.

\begin{proof}
Set $C := \|u_1 - u_0\|_{L^\infty}$. It is straightforward to check that $t \mapsto u_0 - Ct$ and $t \mapsto u_1 - C(1-t)$ are subgeodesics, hence their maximum $v_t := \max(u_0 - Ct, u_1 - C(1-t))$ belongs to $\mathcal S$. This, together with convexity, yields:
\begin{equation}\label{eq: u_boundary_est}
\max(u_0 - Ct, u_1 - C(1-t)) \leq u_t \leq (1-t)u_0 + tu_1.
\end{equation}
Hence, $u \in \PSH(S \times X, \pi^*\omega) \cap L^\infty$ and the boundary data is attained in the uniform topology. That $u$ solves the complex Monge–Ampère equation follows via an adaptation of the classical Perron–Bremermann argument, and uniqueness is a consequence of the comparison principle \cite[Theorem 2.3]{Da17s}.
\end{proof}

When the endpoints $u_0, u_1$ are smooth, second-order estimates for the weak geodesic are available. These were first established in full generality by Chu--Tosatti--Weinkove \cite{CTW}, building on earlier Laplacian estimates of X.X. Chen \cite{Ch} (c.f. \cite{He}). For a thorough discussion and some complementary estimates we refer to  \cite{Bl2,Bl1,BouGeo}.

\begin{theorem}[\cite{CTW, Ch}] \label{thm: ueps_estimates}
Suppose $u_0, u_1 \in \mathcal H_\omega$. Then the boundary value problem \eqref{eq: BVPGeod} admits a solution $u \in \PSH(\overline{S} \times X)$ satisfying the following uniform estimates:
\begin{equation}\label{eq: ueps_estimates}
\| u \|_{C^0(\overline{S} \times X)},\ \| u \|_{C^1(\overline{S} \times X)},\ \| u\|_{C^2(\overline{S} \times X)} \leq C\big(\|u_0\|_{C^4(X)}, \|u_1\|_{C^4(X)}, X, \omega\big).
\end{equation}
\end{theorem}

However, due to the counterexamples of \cite{DL12, LV}, higher-order regularity is known to fail in general. In particular, the geodesic segment joining two smooth Kähler potentials typically fails to be $C^2$, illustrating the limitations of regularity in this infinite-dimensional geometric setting.

An intriguing folklore open question concerns the preservation of strict positivity along $C^{1,\bar{1}}$ geodesic segments:  
\begin{ques}  
Given $u_0, u_1 \in \mathcal{H}_\omega$, does there exist $\varepsilon > 0$ such that  
\[
\omega_{u_t} \geq \varepsilon \, \omega, \qquad t \in [0,1],
\]  
along the $C^{1,\bar{1}}$ geodesic segment $t \mapsto u_t$ connecting $u_0$ and $u_1$?  
\end{ques}  
Partial affirmative results have been obtained by Hu and by Chen--Feldman--Hu \cite{Hu,CFH}, though the findings of \cite{RWN} suggests that the general answer could be negative.

\section{The metric completion} 

Using the theory of currents developed by Bedford--Taylor \cite{BT} and Guedj--Zeriahi \cite{GZ}, one can extend the notion of volume measure $\omega_u^n$ to any potential $u \in \PSH(X,\omega)$.

The resulting measure $\omega_u^n$ is called the non-pluripolar Monge--Amp\`ere measure. Since it does not charge pluripolar sets, typically there might be some loss of total mass:
\[
\int_X \omega_u^n \leq \int_X \omega^n.
\]
Potentials for which equality holds are called \emph{full mass potentials}, and the collection of such functions is denoted by
\[
\mathcal E := \left\{ u \in \PSH(X,\omega) \ \middle|\ \int_X \omega_u^n = \int_X \omega^n \right\}.
\]
This class was introduced by Guedj--Zeriahi (and earlier by Cegrell \cite{Ce98} in the local setting). The finite energy spaces $\mathcal E^p$ are then defined as:
\[
\mathcal E^p := \left\{ u \in \mathcal E \ \middle|\ \int_X |u|^p \omega_u^n < \infty \right\}.
\]

A key feature of these spaces is their \emph{geodesic stability}: if $u_0, u_1 \in \mathcal E^p$, then the weak geodesic $t \mapsto u_t$ connecting them (from \eqref{eq: udef1}) lies entirely within $\mathcal E^p$:

\begin{prop}[\cite{Da17a}]
Let $u_0, u_1 \in \mathcal E^p$. Then $u_t \in \mathcal E^p$ for all $t \in [0,1]$.
\end{prop}

Let $[0,1] \ni t \mapsto \gamma_t \in \mathcal H_\omega$ be a smooth path. The associated $L^p$-path length is defined as:
\[
l_p(\gamma) = \int_0^1 \| \dot \gamma_t \|_p\, dt = \int_0^1 \left( \int_X |\dot \gamma_t|^p \omega_{\gamma_t}^n \right)^{1/p} dt.
\]
This naturally induces a path-length pseudo-metric:
\[
d_p(u,v) := \inf_{\gamma_0 = u,\ \gamma_1 = v} l_p(\gamma).
\]

While the triangle inequality holds by construction, non-degeneracy is a subtle issue in infinite-dimensional geometry. The following theorem confirms that $d_p$ is indeed a metric, and characterizes its metric completion:

\begin{theorem}[\cite{Da15,Da17a}]
\hfill
\begin{itemize}
  \item[(i)] $(\mathcal H_\omega, d_p)$ is a metric space (for $p=2$, this is due to Chen \cite{Ch}).
  \item[(ii)] The metric completion of $(\mathcal H_\omega, d_p)$ is isometric to the finite energy space $\mathcal E^p$.
  \item[(iii)] Given $u_0,u_1 \in \mathcal E^p$, the weak geodesics of \eqref{eq: udef1} are genuine $d_p$-geodesics in the metric completion $\mathcal E^p$.
\end{itemize}
\end{theorem}

Recall that a curve $[0,1] \ni t \mapsto u_t \in \mathcal E^p$ is a $d_p$-geodesic if:
\[
d_p(u_t, u_{t'}) = |t - t'| \cdot d_p(u_0, u_1), \quad \forall t, t' \in [0,1].
\]

In the case $p=2$ the above result was conjectured by V. Guedj, who proved it in the toric case \cite{Guconj}.

Although the definition of $d_p$ involves an infimum over paths, it admits an equivalent and more tractable formulation. Specifically, the following quasi-isometric estimate was established in \cite{Da15}: there exists $C > 1$, depending only on $\dim X$, such that
\begin{equation}\label{eq: d_p_quasi_isom}
\frac{1}{C} d_p(u,v)^p \leq \int_X |u - v|^p \omega_v^n + \int_X |u - v|^p \omega_u^n \leq C\, d_p(u,v)^p, \quad u, v \in \mathcal E^p.
\end{equation}

\paragraph{The $d_1$ metric.} The space $(\mathcal E^1, d_1)$ plays a central role in the study of canonical metrics, and has therefore been investigated in more detail.

This space comes with the Aubin–Yau energy (also called the Monge–Ampère energy), defined as:
\[
I(u) := \frac{1}{V(n+1)} \sum_{j=0}^n \int_X u\, \omega_u^j \wedge \omega^{n-j}, \ u \in \mathcal E^1,
\]
where $V = \int_X \omega^n$ is the total volume. This functional is linear along weak geodesic segments \cite{Ch, Da17, Da15}. Moreover, this functional gives a syntethic formula for the $d_1$ distance \cite{Da15}:
\begin{equation}\label{eq: Pyt}
d_1(u_0,u_1) = I(u_0) + I(u_1) - 2I(P(u_0,u_1)), \ \ u_0,u_1 \in \mathcal E^1,
\end{equation}
where $P(u_0,u_1) \in \mathcal E^1$ is the following rooftop envelope:

$$P(u_0,u_1) := \sup\{v \in \mathcal E^1, \textup{ s.t. } v \leq u_0, \ \ v \leq u_1\}.$$

Unlike the spaces $(\mathcal E^p, d_p)$ for $p > 1$, the $d_1$-geometry is not uniformly convex, and geodesics need not be unique \cite{DL20}. Indeed, due to \eqref{eq: Pyt}, the concatenation of the geodesic segments between $(u_0, P(u_0,u_1))$ and $(P(u_0,u_1), u_1)$ has the same $d_1$-length as the weak geodesic connecting $u_0$ and $u_1$.

Despite this, on the level sets of the $I$ functional the $d_1$ distance has attractive properties. For example, using \eqref{eq: d_p_quasi_isom}, on this set one can compare  $d_1$  to the classical $J$-functional:
\[
J(u) := \frac{1}{V} \int_X u - I(u).
\]
More precisely, there exist constants $m, M, D > 0$ such that:
\[
m\, J(u) - D \leq d_1(0,u) \leq M\, J(u) + D, \ \textup{ for all } \  u \in \mathcal E^1 \textup{ and } I(u)=0.
\]
The values of $m$ and $M$ control the asymptotic growth of the $d_1$ metric. While $M = 2$ is known to be optimal, the sharp value of $m$ remains conjectural outside the toric setting \cite{DGS} (c.f. \cite{BHJ}), where it is known to be:
\[
m = \frac{2}{n+1} \left( \frac{n}{n+1} \right)^n.
\]
We conjecture that this bound is optimal in general:
\begin{conj}
The constant $m = \frac{2}{n+1} \left( \frac{n}{n+1} \right)^n$ is optimal for all Kähler manifolds $(X, \omega)$.
\end{conj}

\section{Geodesic rays} 

Geodesic rays play a fundamental role in the study of canonical metrics, due to their connection with stability notions via asymptotic slope formulas for energy functionals \cite{BBJ,CC2, LiG, DX22}. A systematic investigation of their structure was carried out in \cite{DL20}.

Constructing nontrivial weak geodesic rays is generally challenging. The correspondence introduced by Ross and Witt Nyström \cite{RWN14} provides a flexible and effective tool in this direction. It was later refined in \cite{Da17,DDL3, DX22} etc., and uses the Legendre transform in the time direction to associate geodesic rays to so-called test curves.

For simplicity, we restrict attention to rays of bounded potentials throughout this section.

Fix a potential $u_0 \in \PSH(X,\omega) \cap L^\infty$. A curve 
$$[0,\infty) \ni t \mapsto u_t \in \PSH(X,\omega) \cap L^\infty$$ 
is called a \emph{subgeodesic ray} emanating from $u_0$ if $u_t \to u_0$ in $L^1$-weak topology as $t \to 0$, and the complexification $u(s,x) := u_{\operatorname{Re} s}(x)$ defines a $\pi^* \omega$-plurisubharmonic function on $\{\operatorname{Re} s > 0\} \times X$. For simplicity, we will assume $u_0 = 0$ for the remainder of the discussion.

A subgeodesic ray $t \to u_t$ is said to be \emph{sublinear} if there exists a constant $C > 0$ such that
\[
u_t(x) \leq Ct + C \quad \text{for all } t \geq 0,\ x \in X.
\]

A (bounded) \emph{weak geodesic ray} $t \to u_t$ is a sublinear subgeodesic ray whose complexification satisfies the homogeneous Monge–Ampère equation:
\[
(\pi^* \omega + i \partial \bar \partial u)^{n+1} = 0 \quad \text{on } \{\operatorname{Re} s > 0\} \times X.
\]

\paragraph{The Ross--Witt Nyström correspondence.} We now recall the dual notion of test curves. Following a slight modification of \cite[Definition 5.1]{RWN14}, a map $\mathbb{R} \ni \tau \mapsto \psi_\tau \in \PSH(X,\omega) \cup \{-\infty\}$ is called a \emph{test curve} if:
\begin{itemize}
  \item for each $x \in X$, the function $\tau \mapsto \psi_\tau(x)$ is upper semicontinuous, concave, and decreasing;
  \item $\psi_\tau \equiv -\infty$ for $\tau$ sufficiently positive;
  \item $\psi_\tau \equiv 0$ for $\tau$ sufficiently negative.
\end{itemize}

Throughout this section, we adopt the convention that test curves are parameterized by $\tau$, while geodesic rays are parameterized by $t$. Accordingly, we denote geodesic rays as $\{u_t\}$ and test curves as $\{v_\tau\}$. Thus, the parameters $t$ and $\tau$ can be viewed as dual variables.

Given a test curve $\{\psi_\tau\}$, define the associated thresholds:
\begin{equation}\label{eq: tau_ray}
\tau^-_\psi := \sup\{\tau \in \mathbb{R} \mid \psi_\tau \equiv 0\}, \quad 
\tau^+_\psi := \inf\{\tau \in \mathbb{R} \mid \psi_\tau \equiv -\infty\}.
\end{equation}

Recall from \cite{RWN14, DDL2} that for $\psi, \chi \in \PSH(X,\omega)$, the \emph{envelope of the singularity type of $\chi$ with respect to $\psi$} is defined as:
\begin{equation}\label{eq: sing_env_def}
P[\chi](\psi) := \operatorname{usc}_X \left( \sup \left\{ v \in \PSH(X,\omega) \ \middle| \ v \leq \psi,\ v \leq \chi + C \text{ for some } C \in \mathbb{R} \right\} \right),
\end{equation}
where $\operatorname{usc}_X$ denotes the upper semicontinuous regularization.

Following \cite{DDL2}, $v \in \PSH(X,\omega)$ is called a \emph{model potential} if
\begin{equation}\label{eq: model_def}
v = P[v](0).
\end{equation}
Such potentials play an important role in both the theory of complex Monge–Ampère equations and in the study of plurisubharmonic singularities.

A test curve $\tau \mapsto \psi_\tau$ is said to be \emph{maximal} if all $\psi_\tau, \ \tau \in \Bbb R$ are model, where we used the convention $P(0)[-\infty]  := -\infty$.

The (inverse) Legendre transform of a test curve $\tau \mapsto \psi_\tau$ is defined as:
\[
\check{\psi}_t(x) := \sup_{\tau \in \mathbb{R}} \left( \psi_\tau(x) + t\tau \right), \quad t > 0.
\]
Conversely, the Legendre transform of a subgeodesic ray $t \mapsto \phi_t$ is defined as:
\begin{equation}\label{eq: Leg_transf_de}
\hat{\phi}_\tau(x) := \inf_{t > 0} \left( \phi_t(x) - t\tau \right), \quad \tau \in \mathbb{R}.
\end{equation}

The following theorem establishes the bijective correspondence between maximal test curves and weak geodesic rays, via the above Legendre transforms:

\begin{theorem}[\cite{RWN14, Da17}]\label{thm: max_test_curve_ray_duality}
The (inverse) Legendre transform $\{\psi_\tau\} \mapsto \{\check \psi_t\}$ induces a bijection between maximal test curves and bounded weak geodesic rays emanating from $0$. The inverse is given by $\{\phi_t\} \mapsto \{\hat \phi_\tau\}$.
\end{theorem}

\paragraph{The metric geometry of rays.} Let $\mathcal R$ denote the space of bounded weak geodesic rays emanating from $0$. For $\{u_t^0\}, \{u_t^1\} \in \mathcal R$, the $L^1$ \emph{chordal distance} between these rays is defined as:
\[
d_1^c(\{u_t^0\}, \{u_t^1\}) := \lim_{t \to \infty} \frac{d_1(u_t^0, u_t^1)}{t}.
\]

As shown in \cite{BDL1}, this limit exists, because $t \to d_1(u_t^0, u_t^1) $ is convex. Moreover one can define geodesic segments in $(\mathcal R, d_1^c)$ in the usual metric sense: for each $s \in [0,1]$, there exists a geodesic interpolation $\{u_t^s\} \in \mathcal R$ such that
\[
d_1^c(\{u_t^s\}, \{u_t^{s'}\}) = |s - s'| \cdot d_1^c(\{u_t^0\}, \{u_t^1\}), \ s,s' \in [0,1].
\]
This geodesic segment was constructed via ad hoc methods in \cite{DL20}, and it is natural to seek a more canonical interpretation. In this direction, we formulate the following question:

\begin{question}
Let $[0,1] \ni s \mapsto \{u_t^s\} \in \mathcal R$ denote the $d_1^c$-geodesic connecting $\{u_t^0\}$ and $\{u_t^1\}$. How can this curve be described `canonically'? More specifically, does the curve (or its Legendre transform in $t$) satisfy an extremal property? Does it solve an equation?
\end{question}

Given their significance in K-stability, one is particularly interested in geodesic rays that are approximable by algebraic data. These so-called \emph{approximable} (or maximal) rays $t \to u_t$ are characterized by a regularity property: their Legendre transforms $\hat u_\tau$ are not only model, but also \emph{$\mathcal I$-model} in the sense of \cite{DX22,Tru}. Specifically, $v \in \PSH(X,\omega)$ is said to be $\mathcal I$-model if
\[
v = P[v](0)_\mathcal I,
\]
where the envelope $P[v](0)_\mathcal I$ is defined as:
\begin{equation}\label{eq: sing_env_def_I}
P[v](0)_\mathcal I := \operatorname{usc}_X \left( \sup \left\{ w \in \PSH(X,\omega) \ \middle|\ w \leq 0,\ \mathcal I(cw) \subset \mathcal I(cv) \text{ for all } c>0 \right\} \right).
\end{equation}
Here $\mathcal I(\phi)$ denotes the multiplier ideal sheaf of the quasi-psh function $\phi$. The germs of this sheaf are holomorphic functions $f$ such that $|f|^2e^{-\phi}$ is locally integrable.
 Non-Archimedean characterizations of approximable rays have been given earlier in \cite{BBJ}.

It is natural to ask whether the space of approximable rays is stable under geodesic interpolation:

\begin{conj}
Let $\{u_t^0\}, \{u_t^1\} \in \mathcal R$ be approximable rays, and let $s \mapsto \{u_t^s\}$ be the $d_1^c$-geodesic connecting them. Then for each $s \in [0,1]$, the ray $\{u_t^s\}$ is also approximable.
\end{conj}

Due to work of Blum-Liu-Xu-Zhuang, Finski and  Reboulet \cite{BLXZ, Fin1, Reb} there are now many tools available to tackle this conjecture in the projective setting, but very little is known in the transcendental case.

\paragraph{$C^{1, 1}$ approximation of finite energy rays.} In connection with Donaldson's geodesic stability conjecture, it was shown in \cite[Theorem 1.5]{DL20} that any finite energy geodesic ray with finite $K$ energy slope can be decreasingly approximated by $C^{1,\bar 1}$ geodesic rays. It is natural to ask whether this approximation property holds unconditionally, with full Hessian control:

\begin{question}
Suppose that $\{u_t\} \in \mathcal{R}$. Does there exist a decreasing sequence $\{u_t^k\} \in \mathcal{R}$ with $u_t^k \in C^{1, 1}$ for all $t\geq 0$ and $d_1^c(\{u_t\}, \{u_t^k\}) \to 0 ?$
\end{question}

\section{The initial value problem} 

In this final section, we discuss open questions related to the initial value problem for weak geodesics.

Let $a > 0$, and consider a bounded weak geodesic segment $[0,a] \ni t \mapsto u_t \in \PSH(X,\omega) \cap L^\infty$. Due to the convexity of $t \mapsto u_t(x)$, we may define the initial tangent vector pointwise by:
\[
\dot u_0(x) := \lim_{t \to 0^+} \frac{u_t(x) - u_0(x)}{t}, \quad x \in X.
\]
For simplicity, we assume throughout this section that $u_0 = 0$.

A natural question is whether, given a function $v \in L^\infty(X)$, there exists a weak geodesic $t \mapsto u_t$ with $u_0 = 0$ and $\dot u_0 = v$. However, such an existence question is ill-posed in this infinite-dimensional setting, even when considering weak geodesics of class $C^{1,1}$.

Instead, we focus on the question of \emph{uniqueness} of geodesics determined by their initial tangent vectors:

\begin{ques}\label{conj: IVP_conj}
Let $[0,a] \ni t \mapsto u_t, v_t \in \PSH(X,\omega) \cap L^\infty$ be bounded weak geodesic segments with $u_0 = v_0 = 0$. If $\dot u_0 = \dot v_0$, is it true that $u_t = v_t$ for all $t \in [0,a]$?
\end{ques}

A positive resolution of this problem would demonstrate that weak geodesics do not exhibit the “tree-like” behavior commonly associated with CAT(0) spaces. In the smooth category, an affirmative answer is known under additional regularity assumptions: if both $u_t$ and $v_t$ are non-degenerate and of class $C^3$, the conclusion follows from Monge–Ampère foliation theory \cite{RZ3}. However, the case of $C^{1,1}$ geodesics remains open. \smallskip

Even in the setting of rays (i.e., when $a = \infty$), the problem is far from being  understood. However for rays we can reformulate the problem using Legendre duality. More precisely, let $[0,\infty) \ni t \mapsto u_t, v_t \in \PSH(X,\omega) \cap L^\infty$ be weak geodesic rays satisfying $u_0 = v_0 = 0$ and $\dot{u}_0 = \dot{v}_0$.

By Theorem \ref{thm: max_test_curve_ray_duality}, to conclude that $u_t = v_t$ for all $t \geq 0$, it suffices to show that their Legendre transforms, defined in \eqref{eq: Leg_transf_de}, coincide:
\[
\hat u_\tau = \hat v_\tau, \quad \forall \tau \in \mathbb{R}.
\]

Elementary properties of the Legendre transform yield
\[
\{ \dot u_0 \geq \tau \} = \{ \hat u_\tau = 0 \}, \quad \forall \tau \in \mathbb{R},
\]
and similarly for $v$. Hence, the assumption $\dot u_0 = \dot v_0$ implies
\[
\{ \hat u_\tau = 0 \} = \{ \hat v_\tau = 0 \}, \quad \forall \tau \in \mathbb{R}.
\]
Due to the Ross--Witt Nystr\"om correspondence, $\hat u_\tau$ and $\hat v_\tau$ are model (recall \eqref{eq: model_def}). Consequently, to deduce $\hat u_\tau = \hat v_\tau$, it suffices to show that two model potentials coincide whenever their zero-contact sets coincide, leading to a natural question:

\begin{ques}\label{conj: model}
Let $\phi, \psi \in \PSH(X,\omega)$ be model potentials satisfying $\{\phi = 0\} = \{\psi = 0\}$ together with $\int_X \omega_\phi^n > 0$ and $\int_X \omega_\psi^n > 0$. Is it true that $\phi = \psi$?
\end{ques}
A positive answer to this question would in turn imply a positive answer to Question \ref{conj: IVP_conj} for geodesic rays. In this sense, the uniqueness of weak geodesic rays seems to be intimately connected to the structure of model potentials.

We note that the condition $\{\phi = 0\} = \{\psi = 0\}$ is equivalent with the coincidence of (non-pluripolar) Monge--Amp\`ere measures $\omega_\phi^n = \omega_\psi^n$ \cite[Theorem 1]{DNT}. Thus, equivalently, the above question asks whether the Monge–Ampère measure uniquely determines a model potential.

\footnotesize
\let\OLDthebibliography\thebibliography 
\renewcommand\thebibliography[1]{
  \OLDthebibliography{#1}
  \setlength{\parskip}{1pt}
  \setlength{\itemsep}{1pt}
}

\normalsize
\noindent {\sc Department of Mathematics, University of Maryland}\\
{\tt tdarvas@umd.edu}\vspace{0.1in}

\begin{thebibliography}{0}

\bibitem{BaMa} S. Bando, T. Mabuchi, Uniqueness of Einstein Kähler metrics modulo connected group actions. Algebraic geometry, Sendai, 1985, 11--40, Adv. Stud. Pure Math., 10, North-Holland, Amsterdam, 1987.

\bibitem{BT} E. Bedford, B.A. Taylor, The Dirichlet problem for a complex Monge--Amp\`ere equation, Invent. Math. 37 (1976), 1--44.

\bibitem{BBJ} R. Berman, S. Boucksom, and M. Jonsson. A variational approach to the Yau–Tian–Donaldson conjecture, J. Amer. Math. Soc. 34 (2021), no. 3, 605--652.


\bibitem{BDL1} R. Berman, T. Darvas, C.H. Lu, Convexity of the extended K-energy and the long time behavior of the Calabi flow, Geom. and Topol. 21 (2017), no. 5, 2945--2988.

\bibitem{BDL2} R. Berman, T. Darvas, C.H. Lu, Regularity of weak minimizers of the K-energy and applications to properness and K-stability, Ann. Sci. Éc. Norm. Supér. (4) 53 (2020), no. 2, 267--289.

\bibitem{br2} B. Berndtsson, Probability measures related to geodesics in the space of K\"ahler metrics, arXiv:0907.1806.

\bibitem{Br15} B. Berndtsson, A Brunn-Minkowski type inequality for Fano manifolds and some uniqueness theorems in Kähler geometry. Invent. Math. 200 (2015), no. 1, 149--200.

\bibitem{BrBr} R. Berman, B.  Berndtsson, Convexity of the $K$ energy on the space of Kähler metrics and uniqueness of extremal metrics. J. Amer. Math. Soc. 30 (2017), no. 4, 1165--1196.

\bibitem{Bl1} Z. Błocki, A gradient estimate in the Calabi-Yau theorem, Mathematische Annalen 344 (2009), 317-327.

\bibitem{Bl2}Z. Błocki, The complex Monge-Amp\`ere equation in Kähler geometry, course given at CIME Summer School in Pluripotential Theory, Cetraro, Italy, July 2011, eds. F. Bracci, J. E. Fornæss, Lecture Notes in Mathematics 2075, pp. 95-142, Springer, 2013.

\bibitem{BouGeo} S. Boucksom, Monge-Ampère equations on complex manifolds with boundary,
in: Complex Monge-Ampère equations and geodesics in the space of Kähler metrics, V. Guedj editor. Lecture Notes in Mathematics, 2038. Springer, Heidelberg (2012).
\bibitem{BHJ} S. Boucksom, T. Hisamoto, M. Jonsson, Uniform K-stability, Duistermaat-Heckman measures and singularities of pairs. Ann. Inst. Fourier (Grenoble) 67 (2017), no. 2, 743--841.

\bibitem{BJ25}  S. Boucksom, M. Jonsson, On the Yau-Tian-Donaldson conjecture for weighted cscK metrics, arXiv:2509.15016.

\bibitem{BLXZ} H. Blum, Y. Liu, C. Xu, Z. Zhuang, The existence of the K\"ahler-Ricci soliton degeneration. Forum Math. Pi 11 (2023), Paper No. e9, 28 pp




\bibitem{Cal} E. Calabi, The variation of K\"ahler metrics. I. The structure of the space; II. A minimum problem, Bull. Amer. Math. Soc. 60 (1954), 167--168.

\bibitem{CC} E. Calabi, X.X. Chen, The space of K\"ahler metrics.II., J. Differential Geom. 61 (2002), no. 2, 173-193.

\bibitem{Cala} S. Calamai, The Calabi metric for the space of K\"ahler metrics. Math. Ann. 353 (2012), no. 2, 373--402.
\bibitem{CalZh} S. Calamai, K. Zheng, The Dirichlet and the weighted metrics for the space of Kähler metrics. Math. Ann. 363 (2015), no. 3-4, 817--856.

\bibitem{Ce98} U. Cegrell, Pluricomplex energy, Acta Math. 180 (1998), 187–217.

\bibitem{CC1} X. X. Chen, J. Cheng, On the constant scalar curvature Kähler metrics (I)—A priori estimates. J. Amer. Math. Soc. 34 (2021), no. 4, 909--936.



\bibitem{CC2} X.X. Chen, J. Cheng, On the constant scalar curvature Kähler metrics (II)—Existence results. J. Amer. Math. Soc. 34 (2021), no. 4, 937--1009.
arXiv:1801.00656.


\bibitem{CCX} X.X. Chen, J. Cheng, Y. Xu, A bird's eye view on geodesic problems in the space of Kähler potentials, 55 (2025), no. 4.
Volume 55, Issue 1,

\bibitem{Ch}  X.X. Chen, The space of K\"ahler metrics, J. Differential Geom. 56 (2000), no. 2, 189--234.

\bibitem{CFH} X.X. Chen, M. Feldman, J. Hu, Geodesic convexity of small neighborhood in the space of Kähler potentials. J. Funct. Anal. 279 (2020), no. 7, 108603, 65 pp. 

\bibitem{CS} X.X. Chen,S.  Sun, Space of Kähler metrics (V)—Kähler quantization. Metric and differential geometry, 19--41, Progr. Math., 297, Birkhäuser/Springer, Basel, 2012. 

\bibitem{CTW} J. Chu, V. Tosatti, B. Weinkove, On the $C^{1,1}$ regularity of geodesics in the space of Kähler metrics. Ann. PDE 3 (2017), no. 2, Paper No. 15

\bibitem{ClRb} B. Clarke, Y.A.  Rubinstein, Ricci flow and the metric completion of the space of Kähler metrics. Amer. J. Math. 135 (2013), no. 6, 1477--1505.


\bibitem{Da15} T. Darvas, The Mabuchi geometry of finite energy classes, Adv. Math. 285 (2015), 182-219.

\bibitem{Da17} T. Darvas, Weak geodesic rays in the space of K\"ahler potentials and the class $\mathcal E(X,\omega)$, J. Inst. Math. Jussieu 16 (2017), no. 4, 837--858. arXiv:1307.6822.

\bibitem{Da17a} T. Darvas, The Mabuchi completion of the space of K\"ahler metrics, Amer. J. Math. 139 (2017), no. 5, 1275--1313.

\bibitem{Da17s} T. Darvas, Geometric pluripotential theory on Kähler manifolds, Advances in complex geometry, 1-104, Contemp. Math. 735, Amer. Math. Soc., Providence, RI, 2019.

\bibitem{Da17m} T. Darvas, Metric geometry of normal Kähler spaces, energy properness, and existence of canonical metrics, IMRN (2017), no. 22, 6752-6777.


\bibitem{Da24} T. Darvas, The Mabuchi geometry of low energy classes, Math. Ann. 389 (2024), no. 1, 427-450.


\bibitem{DGS} T. Darvas, E. Goerge, K. Smith, Optimal asymptotic of the J functional with respect to the $d_1$ metric, Selecta Math. (N.S.) 28 (2022), no. 2, Paper No. 43.


\bibitem{DL20} T. Darvas, C.H. Lu, Geodesic stability, the space of rays, and uniform convexity in Mabuchi geometry, Geom. and Topol. 24 (2020), no. 4, 1907--1967. arXiv:1810.04661.

\bibitem{DDL2} T. Darvas,  E. Di Nezza and C. H. Lu, Monotonicity of non-pluripolar products and complex Monge-Ampere equations with prescribed singularity, Analysis \& PDE 11 (2018), no. 8.

\bibitem{DDL3} T. Darvas,  E. Di Nezza and C. H. Lu,  $L^1$ metric geometry of big cohomology classes, Ann. Inst. Fourier (Grenoble) 68 (2018), no. 7, 3053--3086.


\bibitem{DL12} T. Darvas, L. Lempert, Weak geodesics in the space of Kähler metrics, Math. Res. Lett. 19 (2012), no. 5, 1127--1135. 

\bibitem{DLR} T. Darvas, C.H.  Lu, Y.A.  Rubinstein, Quantization in geometric pluripotential theory. Comm. Pure Appl. Math. 73 (2020), no. 5, 1100--1138.

\bibitem{DX22} T. Darvas, M. Xia, The closures of test configurations and algebraic singularity types, Adv. Math. 397 (2022), Paper No. 108198. 


\bibitem{DZ25} T. Darvas, K. Zhang, A YTD correspondence for constant scalar curvature metrics, with K. Zhang, arXiv:2509.15173.

\bibitem{DiG} E. Di Nezza, V. Guedj, Geometry and topology of the space of K\"ahler metrics on singular varieties, Compositio Mathematica, 154 (2018), no. 8, 1593–1632.

\bibitem{DiLu} E. Di Nezza, C.H.  Lu, $L^p$ metric geometry of big and nef cohomology classes. Acta Math. Vietnam. 45 (2020), no. 1, 53--69. 

\bibitem{DNT} E. Di Nezza,  S. Trapani, Monge--Amp\`ere measures on contact sets. Math. Res. Lett. 28 (2021), no. 5, 1337--1352. 



\bibitem{Do} S. K. Donaldson, Symmetric spaces, K\"ahler geometry and Hamiltonian dynamics. In Northern California Symplectic Geometry Seminar, volume 196 of Amer. Math. Soc. Transl. Ser. 2, pages 13--33. Amer. Math. Soc., Providence, RI, 1999.

\bibitem{Gbook} V. Guedj (ed.), Complex Monge–Ampère Equations and Geodesics in the Space of Kähler Metrics, Springer 2012.

\bibitem{Guconj} V. Guedj, The metric completion of the Riemannian space of K\"ahler metrics, arXiv:1401.7857.

\bibitem{GZ} V. Guedj, A. Zeriahi, The weighted Monge-Amp\`ere energy of quasiplurisubharmonic
functions, J. Funct. Anal. 250 (2007), no. 2, 442–482.

\bibitem{Gup24} P. Gupta, Complete geodesic metrics in big classes, Trans. Amer. Math. Soc. 378 (2025), no. 6, 4129--4172.

\bibitem{Fin1} S. Finski, Geometry at infinity of the space of Kähler potentials and asymptotic properties of filtrations. J. Reine Angew. Math. 818 (2025), 115--164.

\bibitem{Fin2} S. Finski, Submultiplicative norms and filtrations on section rings, arXiv:2210.03039 , to appear in Proc. Lond. Math. Soc.

\bibitem{Fin3} S. Finski, Small eigenvalues of Toeplitz operators, Lebesgue envelopes and Mabuchi geometry, arXiv:2502.01554.

\bibitem{He} W. He,  On the space of Kähler potentials. Comm. Pure Appl. Math. 68 (2015), no. 2, 332--343.
\bibitem{Hu} J. Hu, The preservation of convexity by geodesics in the space of Kähler potentials on complex affine manifolds, arXiv:2211.12678 .


\bibitem{LiG} C. Li, Geodesic rays and stability in the cscK problem , Annales Scientifiques de l'ENS (4) 55 (2022), no.6, 1529-1574.


\bibitem{LV} L. Lempert, L. Vivas, Geodesics in the space of K\"ahler metrics. Duke Math. J. 162 (2013), no. 7,1369--1381.

\bibitem{Ma} T. Mabuchi, Some symplectic geometry on compact K\"ahler manifolds I, Osaka J. Math. 24, 1987.


\bibitem{PS1} D. Phong, J. Sturm,  The Monge-Ampère operator and geodesics in the space of Kähler potentials. Invent. Math. 166 (2006), no. 1, 125--149.

\bibitem{PS} D. Phong, J. Sturm,  Lectures on stability and constant scalar curvature. Current developments in mathematics, 2007, 101--176, Int. Press, Somerville, MA, 2009.

\bibitem{RWN14}  J. Ross, D.Witt Nystr\"om, Analytic test configurations and geodesic rays, J. Symplectic Geom. Volume 12, Number 1 (2014), 125–169.
\bibitem{RWN} J. Ross, D.Witt Nystr\"om, On the maximal rank problem for the complex homogeneous Monge-Ampère equation. Anal. PDE 12 (2019), no. 2, 493--503.  

\bibitem{Reb} R. Reboulet, Plurisubharmonic geodesics in spaces of non-Archimedean metrics of finite energy. J. Reine Angew. Math. 793 (2022), 59--103.

\bibitem{Ysurv} Y. A. Rubinstein, An introduction to Kahler geometry: Tian's properness conjectures, in: Geometric Analysis (J. Chen et al., Eds.), Progress in Mathematics, Vol. 333, Birkhauser, 2020, pp. 381-443.
\bibitem{RZ3} Y. A. Rubinstein, S. Zelditch, The Cauchy problem for the homogeneous Monge-Ampère equation, III. Lifespan. J. Reine Angew. Math. 724 (2017), 105--143.

\bibitem{Se} S. Semmes, Complex Monge-Amp`ere and symplectic manifolds, Amer. J. Math. 114 (1992).

\bibitem{St} J. Streets, Long time existence of Minimizing Movement solutions of Calabi flow, Adv. Math. 259 (2014), 688–729.

\bibitem{Sze_book} G. Sz\'ekelyhidi,  An introduction to extremal Kähler metrics. Graduate Studies in Mathematics, 152. American Mathematical Society, Providence, RI, 2014. 

\bibitem{Tru22}A. Trusiani, $L^1$ metric geometry of potentials with prescribed singularities on compact K\"ahler manifolds. J. Geom. Anal. 32 (2022), no. 2, Paper No. 37

\bibitem{Tru}A. Trusiani, K\"ahler-Einstein metrics with prescribed singularities on Fano manifolds. J. Reine Angew. Math. 793 (2022), 1--57.

\bibitem{Yau}
S.T. Yau, On the Ricci curvature of a compact K\"ahler manifold and the complex Monge-Amp\`ere equation. I,
Comm. Pure Appl. Math. 31 (1978), no. 3, 339--411.

\end{thebibliography}
\end{document}